\documentclass[11pt,leqno]{amsart}%
\usepackage{amsmath}
\usepackage{amsfonts}
\usepackage{amssymb}
\usepackage{graphicx}
\usepackage[margin=3.80cm]{geometry}
\providecommand{\U}[1]{\protect\rule{.1in}{.1in}}
\newtheorem{theorem}{Theorem}[section]
\newtheorem*{acknowledgement*}{Acknowledgement}

\newtheorem{lemma}[theorem]{Lemma}

\newtheorem{proposition}[theorem]{Proposition}
\newtheorem{remark}[theorem]{Remark}

\begin{document}
\title[Low index self--shrinkers]{Rigidity and gap results for low index
properly immersed self--shrinkers in $\mathbb{R}^{m+1}$}

\begin{abstract}
In this paper we show that the only properly immersed self--shrinkers 
$\Sigma$ in $\mathbb{R}^{m+1}$ with Morse index $1$ are the hyperplanes through the origin. 
Moreover, we prove that if $\Sigma$ is not a hyperplane through the origin then the index jumps and 
it is at least $m+2$, with equality if and only if $\Sigma$
is a cylinder $\mathbb{R}^{m-k}\times \mathbb{S}^{k}(\sqrt{k})$ for some $1\leq k\leq m-1$.
\end{abstract}

\date{\today}

\author {Debora Impera}
\address{Dipartimento di Matematica e Applicazioni\\
Universit\`a degli Studi di Milano Bicocca\\
via Cozzi 53\\
I-20125 Milano, ITALY}
\email{debora.impera@gmail.com}
\maketitle
\tableofcontents

\section{Introduction}
Let $\Sigma^m$ be a complete connected orientable $m$--dimensional Riemannian manifold without boundary 
isometrically immersed by $x_0:\Sigma^m\to\mathbb{R}^{m+1}$ in the Euclidean space $\mathbb{R}^{m+1}$. We say 
that $\Sigma$ is moved along its mean curvature vector if there is a whole family 
$x_t=x(\cdot\,,\, t)$ of smooth immersions, with corresponding hypersurfaces $\Sigma_t=x_t(\Sigma)$, 
such that it satisfies the mean curvature flow initial value problem
\begin{equation}\label{MCF}
\begin{cases}
\frac{\partial}{\partial t}x(p,t)=H(p,t)\nu(p,t)&p\in \Sigma^m\\
x(\cdot,\,t_0)=x_0.&
\end{cases} 
\end{equation}
Here $H(p,t)$ and $\nu(p,t)$ are respectively the mean curvature and the unit normal 
vector of the hypersurface $\Sigma_t$ at $x(p,t)$. When possible, we will choose the unit normal 
$\nu$ to be inward pointing.

The short time existence and uniqueness of a solution of \eqref{MCF} was investigated in 
classical works on quasilinear parabolic equations. Another interesting and more challenging
question is what happens to these flows in the long term. Classical examples show that 
singularities can happen. A major problem in literature has been to study the nature of 
these singularities and it is a general principle, discovered by Huisken, that the singularities are 
modeled by self--shrinkers.
A connected, isometrically immersed hypersurface
$x: \Sigma^{m}\rightarrow\mathbb{R}^{m+1}$ is said to be a \textit{self}--\textit{shrinker} (based at $0
\in\mathbb{R}^{m+1}$) if the family of surfaces 
$\Sigma_t=\sqrt{-2t}\Sigma$ flows by mean curvature. Equivalently, a self--shrinker
can also be characterized as an isometrically immersed hypersurface whose mean curvature vector field
$\mathbf{H}$ satisfies the equation%
\[
x^{\bot}=-\mathbf{H},
\]
where $\left(  \cdot\right)  ^{\bot}$ denotes the projection on the normal bundle of $\Sigma$.
Note that we are using the convention
$H=\mathrm{Tr}_{\Sigma}A$, where $A$ denotes the second fundamental form of the immersion defined as
\[
A X=-\overline{\nabla}_X\nu,
\]
with $\overline{\nabla}$ Levi--Civita connection of $\mathbb{R}^{m+1}$.
With this convention, the self--shrinker equation takes the scalar form
\[
\left\langle x, \nu \right \rangle= -H.
\]
Standard examples of self--shrinkers are the hyperplanes through the origin of $\mathbb{R}^{m+1}$, 
the sphere $\mathbb{S}^{m}(\sqrt{m})$ and the cylinders 
$\Sigma=\mathbb{R}^{m-k}\times \mathbb{S}^{k}(\sqrt{k})$ for some $1\leq k\leq m-1$.
Other examples of self--shrinkers are due to Angenent, \cite{A}, who constructed 
a family of embedded self--shrinkers that are topologically $\mathbb{S}^1\times\mathbb{S}^{m-1}$.

It is well--known that self--shrinkers in $\mathbb{R}^{m+1}$ can be viewed as 
$f$--minimal hypersurfaces, 
that is, critical points of the weighted area functional
\[
\ \mathrm{vol}_f(\Sigma)=\int_{\Sigma}e^{-f}d\mathrm{vol}_{\Sigma},
\]
where $f=|x|^2/2$ (we refer the reader to the papers 
\cite{ChMeZh1}, \cite{Esp}, \cite{IR} for more details on $f$--minimal hypersurfaces).
Moreover, we say that a self--shrinker is $f$--stable if it is a local minimum of the weighted area 
functional for every compactly supported normal variation. In the instability case, it makes sense to investigate the Morse index, 
that is, roughly speaking, the maximum dimension of the linear space 
of compactly supported deformations that decrease the weighted area up to second order. 

It was proved by Colding and Minicozzi, \cite{CoMi}, that every complete properly immersed 
self--shrinker is necessarily $f$--unstable. Equivalently, every properly immersed self--shrinker 
has Morse index greater than or equal to $1$. In the equality case, rigidity 
results have been proved by Hussey, \cite{Hu}, under the additional assumption of embeddedness. 
More precisely, he showed that if a complete properly embedded self--shrinker in 
$\mathbb{R}^{m+1}$ has Morse index $1$, then it has to be a hyperplane through the origin. 
Furthermore, he also proved that if the self-shrinker is not a hyperplane through the origin, 
then the Morse index jumps and it has to be at least $m+2$, with equality if and only if 
the self--shrinker is a cylinder $\mathbb{R}^{m-k}\times \mathbb{S}^{k}(\sqrt{k})$ 
for some $1\leq k\leq m$. It is worth pointing out that 
there many examples of non-embedded self--shrinkers. Indeed, Abresch and Langer, \cite{AL}, 
constructed self--intersecting curves in the plane that are self--shrinkers under the mean curvature 
flow. Moreover, it is possible to obtain examples of non-embedded self-shrinkers by taking products 
of these curves with Euclidean factors. Finally Drugan and Kleene, \cite{DK}, recently constructed 
in generic dimension infinitely many immersed rotationally symmetric self-shrinkers having 
the topological type of the sphere $(S^n)$, the plane $(\mathbb{R}^{n})$, the cylinder
$(\mathbb{R}\times\mathbb{S}^{n-1})$, and the torus $(\mathbb{S}^1\times\mathbb{S}^{n-1})$.

The aim of this paper is to investigate if hyperplanes through the origin and cylinders remains the only 
hypersurfaces, among the wider family of properly immersed self--shrinkers 
in $\mathbb{R}^{m+1}$, having Morse index $1$
and $m+2$ respectively, and if, except for them, every properly immersed self--shrinker has Morse index strictly bigger
than $m+2$. 
Exploiting the link between stability properties of self--shrinkers and spectral properties of 
a suitable weighted Schr\"{o}dinger operator, as well as some basic identities 
which naturally involve the weighted Laplacian of the self--shrinker,
we provide a positive answer to the above mentioned problem. More precisely, 
we prove the following
\begin{theorem}\label{thm_Index}
Let $\Sigma^m$ be a complete properly immersed self--shrinker in $\mathbb{R}^{m+1}$. Then 
\begin{enumerate}
 \item $\mathrm{Ind}_f(\Sigma)\geq1$ and equality holds if and only if $\Sigma$ is a hyperplane through the origin;
 \item If $\Sigma$ is non--totally geodesic, then $\mathrm{Ind}_f(\Sigma)\geq m+2$. Moreover, 
$\mathrm{Ind}_f(\Sigma)=m+2$ if and only if $\Sigma=\mathbb{R}^{m-k}\times \mathbb{S}^{k}(\sqrt{k})$ 
for some $1\leq k\leq m-1$.
\end{enumerate}
\end{theorem}

\section{Some spectral theory and potential theory on weighted manifolds}\label{sec_spectral}
\subsection{Some spectral theory for weighted Schr\"odinger operators}
A weighted manifold is a triple 
$\Sigma^{m}_f=(\Sigma^{m}, \left\langle \,,\,\right\rangle, e^{-f}d\mathrm{vol}_\Sigma)$, 
where $(\Sigma^{m}, \left\langle \,,\,\right\rangle)$ is a complete $m$--dimensional Riemannian manifold, 
$f\in C^{\infty}(\Sigma)$ and $d\mathrm{vol}_\Sigma$ denotes the canonical Riemannian volume form on 
$\Sigma$. In the following we collect some well-know facts about spectral theory on weighted 
manifolds (see e.g. \cite[Chapter 3]{PRS_Progress} for an exhaustive survey on spectral 
theory on Riemannian manifolds).

Associated to a weighted manifold $\Sigma_f$ 
there is a natural divergence form second order diffusion operator,the $f$--Laplacian,
defined on $u$ by
\[
\Delta_{f}u=e^{f}\operatorname{div}\left(  e^{-f}\nabla u\right)=
\Delta u -\left\langle \nabla u, \nabla f\right\rangle.
\]
This is clearly symmetric on $L^{2}(\Sigma_f)$ endowed with the inner product
\[
(u,v)_{L^2(\Sigma_f)}=\int_{\Sigma} uv e^{-f}d\mathrm{vol}_{\Sigma}.
\]
Given $q\in L^{\infty}_{loc}(\Sigma)$,
consider the \textit{weighted Schr\"odinger operator}
\[
Lu=-\Delta_fu-qu,\quad \forall u\in C^{\infty}_c(\Sigma).
\]
This is again a symmetric linear operator on $L^2(\Sigma_f)$ and we set 
$Q$ to be the symmetric bilinear form on $L^2(\Sigma_f)$ defined as 
$Q(u,v):=(L u,v)_{L^2(\Sigma_f)}$. 
Recall that $L$ is said to be \textit{bounded from below} if 
$$Q(u,u)\geq c\|u\|^2_{L^2(\Sigma_f)},\quad c\in \mathbb{R}.$$ 
In particular, when $c\geq0$,  $L$ is said to be \textit{non--negative}. 

Given any open, relatively compact domain $\Omega\subset\Sigma$ we define $L_{|\Omega}$ to be 
the operator $L$ acting on $C^{\infty}_c(\Omega)$ and we denote by $L_\Omega$ its Friedrichs 
extension. By standard spectral theory, $L_\Omega$ has purely discrete spectrum consisting
of a divergent sequence of eigenvalues $\{\lambda_k(L_\Omega)\}$.
The first eigenvalue of $L_\Omega$ is defined by Rayleigh characterization as
\[
\lambda_1(L_\Omega)=\inf_{0\neq u\in C^{\infty}_c(\Omega)}
\frac{Q(u,u)}{\|u\|^2_{L^2(\Sigma_f)}}.
\]
Moreover, we define the index of $L_\Omega$, $\mathrm{Ind}(L_\Omega)$, to be the
number, counted according to multiplicity,
of negative eigenvalues of $L_\Omega$. \\
The \textit{bottom of the spectrum of} $L$ on $\Sigma$ is then defined as 
\[
\lambda_1^L(\Sigma)=\inf\{\lambda_1(L_\Omega)\ :\ \Omega\subset\subset \Sigma\}.
\]
Similarly, the \textit{Morse index of} $L$ on $\Sigma$ is defined as
\[
\mathrm{Ind}^L(\Sigma):=\sup\{\mathrm{Ind}(L_\Omega)\ :\ \Omega\subset\subset\Sigma \}.
\]
Adapting to the weighted setting arguments in \cite{FC} it is not difficult to prove the following
\begin{proposition}\label{prop_basisL}
Let $\Sigma_f$ be a weighted manifold and let $L=-\Delta_f-q$, $q\in L^{\infty}_{loc}(\Sigma)$.
The following are equivalent:
\begin{enumerate}
 \item $\mathrm{Ind}^L(\Sigma)<+\infty$;
 \item There exists a finite dimensional subspace $W$ of the weighted space $L^2(\Sigma_f)$ having an
 orthonormal basis $\psi_1,\cdots,\psi_k$ consisting of eigenfunctions of $L$ with eigenvalues 
 $\lambda_1,\cdots,\lambda_k$ respectively. Moreover, each $\lambda_i$ is negative and any function
 $\phi\in C^{\infty}_c(\Sigma)\cap W^{\bot}$ satisfy $Q(\phi,\phi)\geq 0$.
 \end{enumerate}
\end{proposition}

Finally, in case $L$ is essentially self--adjoint, we can relate the Morse index 
to the so--called spectral index of $L$.
Towards this aim we first recall that, given a self--adjoint operator 
$T:\mathcal{D}(T)\rightarrow L^2(\Sigma_f)$, its spectral index is defined as
\[
\widetilde{\mathrm{Ind}}^{T}(\Sigma)=
\sup\left\{\mathrm{dim}V\ :\ V\subset \mathcal{D}(T), 
\ (T u,u)_{L^2(\Sigma_f)}<0\ \forall\ 0\neq u\in V  \right\}.
\]
If $L$ is essentially self--adjoint, then there is a unique 
self--adjoint extension $L_\Sigma$ of $L$ and we can define the \textit{spectral index of} $L$
as the spectral index of its self--adjoint extension, that is
\[
\widetilde{\mathrm{Ind}}^{L}(\Sigma):=\widetilde{\mathrm{Ind}}^{L_\Sigma}(\Sigma).
\]
Furthermore, since $\Sigma$ is complete, if $L$ is 
also bounded from below by a constant $c$, then it is essentially self--adjoint. In this case 
$L_\Sigma$ corresponds to the Friedrichs extension of $L$, that is, the 
self--adjoint extension of $L$ associated to the closure of the quadratic form
$Q$ with respect to the norm $\| \cdot \|_Q$ induced by the inner product 
$Q(\ ,\ )+(1-c)(\ ,\ )_{L^2(\Sigma_f)}$. Moreover, it turns out that
\[
\overline{C^{\infty}_c(\Sigma)}^{\|\cdot \|_Q}=
\{u\in W^{1,2}(\Sigma_f): |q|^{1/2}u\in L^2(\Sigma_f)\}=:\mathcal{V}_q,
\]
and hence the domain of the operator $L_\Sigma$ is the space 
\[
\mathcal{D}(L_\Sigma)=\{u\in \mathcal{V}_q\ : \ L_\Sigma u \in L^2(\Sigma_f)\},
\]
where $L_\Sigma u$ is understood in distributional sense.

The relationship between the two concepts of index presented above is clarified by the following 
\begin{theorem}\label{thm_eqfinitind}
Let $\Sigma_f$ be a weighted manifold and let $L=-\Delta_f-q$, $q\in L^{\infty}_{loc}(\Sigma)$.
\begin{itemize}
 \item[(i)] If $L$ is essentially self--adjoint on $C_c^{\infty}(\Sigma)$, 
 then $\widetilde{\mathrm{Ind}}^{L}(\Sigma)=\mathrm{Ind}^{L}(\Sigma)$;
 \item[(ii)] if $\mathrm{Ind}^L(\Sigma)<+\infty$, then $L$ is bounded from below, essentially
 self--adjoint on $C_c^{\infty}(\Sigma)$ and 
$\mathrm{Ind}^L(\Sigma)=\widetilde{\mathrm{Ind}}^L(\Sigma)<+\infty$.  
\end{itemize}
\end{theorem}
The proof of the previous theorem is a consequence of 
Theorem 3.17 in \cite{PRS_Progress} taking into account 
that $L=-\Delta_f-q\left(x\right)$ is unitarily equivalent to the Schr\"odinger operator 
\[
\ S=-\Delta-\left[\left(1/4\left\langle \nabla f,\nabla f\right\rangle-1/2\Delta f\right)-
q\left(x\right)\right]=-\Delta-\left(p\left(x\right)+q\left(x\right)\right)
\]
under the multiplication map 
$T(u)=e^{-f/2}u$ of $L^2(\Sigma)$ onto $
L^2(\Sigma_f)$ (see for instance \cite{Setti_RSMUP}).

\subsection{f--parabolicity of weighted manifolds}
Following classical terminology in linear potential theory we say that
a weighted manifold $\Sigma_f$ is $f$--\textit{parabolic} if
\[
\begin{cases}
\Delta_{f}u\geq0\\
u^{\ast}=\sup_{M}u<+\infty
\end{cases}
\quad \Rightarrow \quad u\equiv u^{\ast}.
\]
As a matter of fact, $f$--parabolicity is related to a wide class of equivalent properties
involving the recurrence of the Brownian motion, $f$--capacities of condensers, 
the heat kernel associated to the drifted laplacian, weighted volume growth, 
function theoretic tests, global divergence theorems
and many other geometric and potential-analytic properties. 
Here we limit ourselves to point out the following characterization.

\begin{theorem}\label{thm_KNR}
A weighted manifold $\Sigma_f$ is $f$--parabolic if and only if for every 
vector field $X$ satisfying
\begin{itemize}
\item[(i)] $|X|\in L^2(\Sigma_f)$,
\item[(ii)] $(\mathrm{div}_fX)_-\in L^{1}_{loc}(\Sigma_f)$
\end{itemize}
it holds
\[
\int_\Sigma \mathrm{div}_f(X)e^{-f}d\mathrm{vol}_\Sigma=0,
\]
where
$$\mathrm{div}_f(X)=e^f\mathrm{div}(e^{-f}X).$$
\end{theorem}
We refer the reader to \cite[Theorem 7.27]{PRS_Progress} for a detailed proof of the 
previous result in the unweighted setting. Although the proof of this theorem can be deduced 
adapting to the weighted laplacian $\Delta_{f}$ the proof in \cite{PRS_Progress}, 
we provide here a shorter and more direct proof.
\begin{proof}[Proof of Theorem \ref{thm_KNR}]
We consider the warped product $\overline{\Sigma}=\Sigma\times_h\mathbb{T}$, 
where $h:=e^{-f}$ and $\mathbb{T}=\mathbb{R}/\mathbb{Z}$, so that $\mathrm{vol}(\mathbb{T})=1$. 
We first note that, as proven in \cite[Lemma 2.6]{RV}, the weighted manifold  $\Sigma_f$ is 
$f$--parabolic if and only if $\overline{\Sigma}$ is parabolic. Moreover, as a consequence 
of the Kelvin--Nevanlinna--Royden criterion (see \cite[Theorem 7.27]{PRS_Progress}), 
the parabolicity of $\overline{\Sigma}$ is equivalent to the fact 
that, for every vector field $Y\in T\overline{\Sigma}$ satisfying
\begin{itemize}
 \item[(a)] $|Y|\in L^2(\overline{\Sigma})$;
 \item[(b)] $(\overline{\mathrm{div}})_-(Y)\in L^1_{loc}(\overline{\Sigma})$
\end{itemize}
it holds
\[
\int_{\overline{\Sigma}}\overline{\mathrm{div}}(Y)d\mathrm{vol}_{\overline{\Sigma}}=0.
\]
Here we have denoted by $\overline{\mathrm{div}}$ the divergence with respect to the metric
$g_{\overline{\Sigma}}=\pi_\Sigma^*(g_\Sigma)+h^2\pi_{T}^*(dt^2)$, where $\pi_\Sigma$ and
$\pi_\mathbb{T}$ denote the projections of $\overline{\Sigma}$ onto $\Sigma$ and $\mathbb{T}$
respectively. Given a vector field $Y\in T\overline{\Sigma}$, we let 
$X=(\pi_\Sigma)_*(Y)$. Taking into account
that $d\mathrm{vol}_{\overline{\Sigma}}=e^{-f}d\mathrm{vol}_\Sigma dt$ and using the formulas
for covariant derivatives on warped products (see \cite{O}), it is not difficult
to prove that 
\begin{align*}
\mathrm{div}_f( X) &=\overline{\mathrm{div}}(Y),\\
\int_\Sigma |X|^2 e^{-f}d\mathrm{vol}_\Sigma&\leq 
\int_{\overline{\Sigma}}|Y|^2 d\mathrm{vol}_{\overline{\Sigma}}\\
\int_\Sigma (\mathrm{div}_f)_- (X) e^{-f}d\mathrm{vol}_\Sigma&= 
\int_{\overline{\Sigma}}(\overline{\mathrm{div}})_-(Y) d\mathrm{vol}_{\overline{\Sigma}}.
\end{align*}
Now assume that $\Sigma$ is $f$--parabolic and let $X\in T\Sigma$ be a vector field satisfying 
the integrability conditions $(i)$ and $(ii)$. Define $Y\in T\overline{\Sigma}$ by $Y_{(x,t)}=X_x$.  
Then $(\pi_\Sigma)_*Y=X$, $|X|^2=|Y|^2$ and $Y$ satisfies the integrability conditions 
$(a)$ and $(b)$. Moreover, since $\overline{\Sigma}$
is parabolic, it holds
\[
0=\int_{\overline{\Sigma}}\overline{\mathrm{div}}(Y)d\mathrm{vol}_{\overline{\Sigma}}
=\int_\Sigma \mathrm{div}_f(X)e^{-f}d\mathrm{vol}_\Sigma, 
\]
proving that $f$--parabolicity is a sufficient condition for the validity of the global (weighted)
version of the Stokes theorem.

As for the converse, assume that for every vector field 
$X \in T\Sigma$ satisfying the integrability conditions $(i)$ and $(ii)$ it holds
\[
\int_\Sigma \mathrm{div}_f(X)e^{-f}d\mathrm{vol}_\Sigma=0.
\]
Suppose by contradiction that $\Sigma$ is not $f$--parabolic.
Thus $\overline{\Sigma}$ is not parabolic and there exists $Y\in T\overline{\Sigma}$ satisfying the
integrability conditions in $(a)$ and $(b)$ and such that
\[
\int_{\overline{\Sigma}}\overline{\mathrm{div}}(Y)d\mathrm{vol}_{\overline{\Sigma}}\neq0. 
\]
Set $X=(\pi_\Sigma)_*(Y)$. Then $X$ 
satisfies the integrability conditions $(i)$ and $(ii)$. However
\[
\int_\Sigma \mathrm{div}_f(X)e^{-f}d\mathrm{vol}_\Sigma=
\int_{\overline{\Sigma}}\overline{\mathrm{div}}(Y) d\mathrm{vol}_{\overline{\Sigma}}\neq 0,
\]
leading to a contradiction.
\end{proof}

From the geometric point of view, it is well known
that $f$--parabolicity is related to the growth rate of 
the weighted volume of intrinsic metric objects. 
Indeed, adapting to 
the diffusion operator $\Delta_{f}$ standard proofs for the Laplace--Beltrami operator 
(see for instance \cite{Gr}, \cite{RS}), one can prove the following
\begin{proposition}
Let $\Sigma_f$ be a weighted manifold. If
\begin{equation}\label{eq_fparspheres}
\mathrm{vol}_{f}\left(  \partial B_{r}(o)\right)  ^{-1}\notin L^{1}\left(
+\infty\right),
\end{equation}
then $\Sigma_f$ is $f$--parabolic.
\end{proposition}
Here $\partial B_r(o)$ denotes the geodesic ball of radius $r$ centered at a reference point 
$o\in\Sigma$ and
\[
\mathrm{vol}_{f}\left(  \partial B_{r}\right)=\int_{\partial B_r}e^{-f}d\mathrm{vol}_{m-1},
\]
where $d\mathrm{vol}_{m-1}$ denotes the $(m-1)$--dimensional Hausdorff measure.

Observe also that if $\mathrm{vol}_f(\Sigma)<+\infty$, then 
condition \eqref{eq_fparspheres} is automatically satisfied. Hence we conclude that any
weighted manifold $\Sigma_f$ satisfying $\mathrm{vol}_f(\Sigma)<+\infty$ is $f$--parabolic.
\begin{remark}\label{rmk_fparss}
\rm{
It was proved in \cite{ChZh} and \cite{DX}
that for any complete immersed self-shrinker $x:\Sigma^m\to\mathbb{R}^{m+1}$ the following statements are 
equivalent:
\begin{itemize}
	\item[(a)] the immersion $x$ is proper;
	\item[(b)] $\Sigma$ has extrinsic polynomial volume growth;
	\item[(c)] $\Sigma$ has extrinsic Euclidean volume growth;
	\item[(d)] $\Sigma$ has finite $f$--volume, where $f=|x|^2/2$.	
\end{itemize}
In particular, the equivalence $(a)$--$(d)$ shows that any complete self--shrinker properly 
immersed in $\mathbb{R}^{m+1}$ is $f$--parabolic, with $f=|x|^2/2$.
}
\end{remark}

\section{Characterization of low index properly immersed self--shrinkers}
Let $x:\Sigma^m\to\mathbb{R}^{m+1}$ be a complete self-shrinker and set $f=\frac{|x|^2}{2}$.
The function $f$ induces a weighted structure on the self--shrinker that can hence be viewed 
as a weighted manifold itself. 
Basic geometric quantities on the self-shrinker satisfy 
identities which naturally involve the $f$-Laplacian on 
$\Sigma_{f}$, as shown in the next
\begin{proposition}\label{prop_eigenjacobi}
Let $\Sigma^m$ be a self--shrinker in $\mathbb{R}^{m+1}$ and let 
$a\in\mathbb{R}^{m+1}$ be a constant vector. Then 
\begin{enumerate}
 \item the mean curvature $H$ satisfies
\[
\Delta_f H= (1-\| A \|^2)H.
\]
 \item The function $g_a=\langle \nu, a\rangle$ satisfies
\[
\Delta_f g_a=- \| A \|^2 g_a.
\]
 \item The function $l_a=\langle x, a\rangle$ satisfies
\[
\Delta_f l_a=- l_a.
\]
\item the squared norm of the second fundamental form satisfies the Simons type formula
\[
\Delta_f \|A\|^2=2|\nabla A|^2+2\|A\|^2(1-\|A\|^2).
\]
\end{enumerate}
\end{proposition}
We refer the reader to \cite[Theorem 4.1]{Huisken} and \cite[Theorem 5.2]{CoMi} 
for a proof of the identities listed above.

It turns out that stability properties of self--shrinkers, viewed as
critical points of the weighted area functional, are taken into account by spectral
properties of the \textit{weighted Jacobi operator} $L_f$,
defined as
\[
L_f u=-\Delta_f u-(\|A\|^2+1)u.
\]
To be more precise, we say that a self--shrinker $\Sigma$ is $f$--\textit{stable}
if and only if the operator $L_f$ is non--negative, that is if and only if
\[
\lambda_1^{L_f}(\Sigma)\geq 0.
\]
Furthermore, we define the $f$--Index of $\Sigma$ to be the Morse index of the 
Jacobi operator $L_f$, that is
\[
\mathrm{Ind}_f(\Sigma):=\sup\{\mathrm{Ind}((L_f)_\Omega)\ :\ \Omega\subset\subset\Sigma \}.
\]
\begin{remark}\label{rmk_domainQ}
\rm{Keeping in mind Theorem \ref{thm_eqfinitind} we see that if $\mathrm{Ind}_f(\Sigma)<+\infty$, 
then the weighted Jacobi operator $L_f$ is bounded from below, it is essentially
self--adjoint and 
$\mathrm{Ind}_f(\Sigma)=\widetilde{\mathrm{Ind}}^{L_f}(\Sigma)$. Moreover, the domain of 
the quadratic form $Q(\cdot,\cdot)=(L_f\cdot,\cdot)_{L^2(\Sigma_f)}$ is the space
\[
\mathcal{V}=\{u\in W^{1,2}(\Sigma_f)\ :\ \|A\|u\in L^2(\Sigma_f)\}. 
\]
Furthermore, we point out that, if $u$ is an eigenfunction of the Jacobi operator $L_f$ and $u\in W^{1,2}(\Sigma_f)$,
then Lemma 9.15 in \cite{CoMi} implies that $u\in \mathcal{V}$.
Finally, if $\mathrm{Ind}_f(\Sigma)<+\infty$, then $\lambda_1^{L_f}(\Sigma)>-\infty$ and 
one can prove the existence of a positive $C^2$ function $v$ satisfying
$L_f v=\lambda_1^{L_f}(\Sigma) v$ (see Lemma 9.25 in \cite{CoMi}). Then, applying again Lemma 9.15 in \cite{CoMi}, it is 
straightforward to prove that any function $\phi \in W^{1,2}(\Sigma_f)$ belongs to $\mathcal{V}$.
} 
\end{remark}
In the following we collect some lemmas that will be essential for the proof of the 
main result of the paper.

\begin{lemma}\label{lemma_dimV}
Let $\Sigma^m$ be a properly immersed non--totally geodesic self--shrinker in 
$\mathbb{R}^{m+1}$ satisfying $\mathrm{Ind}_f(\Sigma)<+\infty$.
Set
\[
W := \{g_b:\Sigma\rightarrow \mathbb{R},\ g_b(p)=\langle \nu(p), b\rangle 
\ \forall p\in \Sigma,\ b\in\mathbb{R}^{m+1}\}.
\]
Then $\mathrm{dim}W\leq m+1$ and, if $\mathrm{dim}W=k<m+1$, we can find $m+1$ linearly independent 
constant non--null vectors $\{b_1,\ldots,b_{m+1}\}\in\mathbb{R}^{m+1}$ satisfying the following properties:
\begin{enumerate}
\item $W=\mathrm{span}\{g_{b_1},\ldots,g_{b_k}\}$;
\item Set $U:=\mathrm{span}\{g_{b_1},\ldots,g_{b_k},l_{b_{k+1}}H,\ldots,l_{b_{m+1}}H\}$.
Then 
\begin{enumerate}
 \item the functions $l_{b_j}H$, $j=k+1,\ldots,m+1$, are eigenfunctions of $L_f$ 
corresponding to the eigenvalue $-1$;
 \item $\mathrm{dim} U=m+1$;
 \item $U\subset \mathcal{V}$.
\end{enumerate}
Moreover, having set
\[
V:=\{\varphi:\Sigma\rightarrow \mathbb{R},\ \varphi(p)=a+\phi(p) 
\ \forall p\in \Sigma,\ a\in\mathbb{R},\
\phi\in U\}\subset \mathcal{V},
\]
one has $\mathrm{dim}V=m+2$. 
\end{enumerate}
\end{lemma}
\begin{proof}
We note first that it is not difficult to prove that if $\mathrm{dim}W=k<m+1$, then we can find 
$m+1$ linearly independent 
constant non--null vectors $\{b_1,\ldots,b_{m+1}\}\in\mathbb{R}^{m+1}$ such that
$W=\mathrm{span}\{g_{b_1},\ldots,g_{b_k}\}$ and $g_{b_j}\equiv 0$ for any $j=k+1,\cdots,m+1$.
This proves part $(1)$. As for part $(2a)$, it 
suffices to show that the functions $l_{b_j}H$, 
$j=k+1,\cdots,m+1$, satisfy $L_f l_{b_j}H=-l_{b_j}H$. 
Note that $g_{b_j}\equiv 0$, a simple computation shows that 
the functions $l_{b_j}$ satisfy
\[
\nabla l_{b_j}=b_j^T=b_j,\quad A b_j\equiv 0.
\]
Thus
\begin{align*}
\Delta_f H l_{b_j}&=H\Delta_f l_{b_j}+2\langle\nabla H, \nabla l_{b_j}\rangle+l_{b_j}\Delta_f H\\
&=-H l_{b_j}-2 \langle AX^T,b\rangle+H l_{b_j}(1-\|A\|^2)\\
&=-\|A\|^2H l_{b_j}, 
\end{align*}
showing that the functions $l_{b_j}H$, $j=k+1,\cdots,m+1$, are eigenfunctions of $L_f$ 
corresponding to the eigenvalue $-1$. 

As for part $(2b)$, let us prove first that the functions $g_{b_1},\ldots,\ g_{b_k}$, 
$l_{b_{k+1}}H,\ldots,\ l_{b_{m+1}}H$ 
are linearly independent.
It suffices to show that, for any non-null vectors $b=\alpha_1b_1+\cdots+\alpha_k b_k$, 
$c=\alpha_{k+1}b_{k+1}+\ldots+\alpha_{m+1} b_{m+1}$, the functions 
$g_b$ and $l_c H$ can not be linearly dependent. Indeed, assume by contradiction that there exists a non--zero
constant $\lambda$ such that $g_b=\lambda l_c H$. Then, in particular
\[
-Ab^T=\nabla g_b=\lambda l_c\nabla H+\lambda H\nabla l_c=-\lambda l_c Ax^T+\lambda H c.
\]
In particular, 
\[
0=-\langle A b^T, c\rangle=-\langle\lambda l_c Ax^T, c\rangle+\lambda H |c|^2=\lambda H |c|^2.
\]
Hence $H\equiv 0$ and $\Sigma$ must be a hyperplane through the origin, which is absurd.

In order to prove that $U\subset\mathcal{V}$ it suffices to show, taking into account Remark 
\ref{rmk_domainQ}, 
that $g_{b_i},\ l_{b_j}H\in W^{1,2}(\Sigma_f)$, for any $i=1,\cdots,k$, $j=k+1,\cdots,m+1$.
Clearly, since $\Sigma$ has finite $f$-volume, the Cauchy-Schwartz inequality implies that
the functions $g_{b_i}$ belong to $L^2(\Sigma_f)$. Moreover, since $\Sigma$ has finite $f$--index, 
Theorem 2 in \cite{IR} implies that $\|A\|\in L^2(\Sigma_f)$ and hence, in particular, 
$|\nabla g_{b_i}|\leq |b_i|\| A\|\in L^2(\Sigma_f)$.  
As for the functions $l_{b_j}H$, note that Lemma 25 in \cite{PiRi} implies that any properly 
immersed self-shrinker satisfies $\mathcal{P}(|x|)\in L^{1}(\Sigma_f)$,
with $\mathcal{P}(t)$ any polynomial in $t$. In particular, the function 
$|x|^{2q}$ belongs to the space $W^{1,2}(\Sigma_f)$ for any $q\in \mathbb{N}$ and hence, taking into account Remark \ref{rmk_domainQ}, 
$|x|^{2q}\|A\|\in L^2(\Sigma_f)$. Thus, using the self-shrinker equation, 
it is straightforward to see that
\begin{align*}
|l_{b_j}||H|&\leq |b_j||x||H|\leq |b_j||x|^2\in L^2(\Sigma_f)\\
|\nabla (l_{b_j}H)|&\leq |b_j||x|^2\|A\|+|b_j||x|\in L^{2}(\Sigma_f).
\end{align*}

Finally, it only remains to prove that $\mathrm{dim}V=m+2$. Assume by contradiction that
there exists a non-zero constant $a$ and a function $\phi\in U$ such that 
$\phi=a$. Then
\[
0=\Delta_f \phi=-\|A\|^2\phi=-\|A\|^2 a,
\]
contradicting again the assumption of $\Sigma$ being non--totally geodesic.
\end{proof}
\begin{lemma}\label{prop_ineqsff}
Let $\Sigma^m$ be a properly immersed self--shrinker in $\mathbb{R}^{m+1}$ satisfying $\mathrm{Ind}_f(\Sigma)<+\infty$.
Let $\phi\in U$, where $U$ is defined as in the previous Lemma.
Then
\[
\int_\Sigma \phi\| A\|^2e^{-f}d\mathrm{vol}_\Sigma=0.
\] 
\end{lemma}

\begin{proof}
Note that the assumption of $\Sigma$ being properly immersed implies that it is 
$f$--parabolic. Let $\phi\in U$. Then $|\nabla \phi|\in L^2(\Sigma_f)$ 
and the conclusion follows as an application 
of Theorem \ref{thm_KNR} to the vector field $X=\nabla \phi$, 
keeping in mind Proposition \ref{prop_eigenjacobi} and Lemma \ref{lemma_dimV}.
\end{proof}

We are now ready to prove the main theorem of this paper.
\begin{proof}[Proof of Theorem \ref{thm_Index}]
Note first that if $\mathrm{Ind}_f(\Sigma)=+\infty$, then the inequality
$\mathrm{Ind}_f(\Sigma)\geq m+2$ is trivially satisfied. Hence
assume that $\mathrm{Ind}_f(\Sigma)<+\infty$.
We claim that any function $\varphi=a+\phi$, for some $a\in\mathbb{R}$, $\phi\in U$, 
satisfies
\[
\int_\Sigma \varphi L_f\varphi e^{-f}d\mathrm{vol}_\Sigma=-\int_\Sigma \varphi\Delta_f \varphi+(\|A\|^2+1)\varphi^2e^{-f}d\mathrm{vol}_\Sigma<0.
\]
In this case, Lemma \ref{lemma_dimV} would imply that either $\Sigma$ is totally geodesic 
(and hence a hyperplane through the origin) or
$\mathrm{Ind}_f(\Sigma)\geq \mathrm{dim}V=m+2$.

A straightforward computation shows that
\[
\varphi\Delta_f\varphi+(\|A\|^2+1)\varphi^2=\varphi^2 +\|A\|^2a^2+a\phi\|A\|.
\]
Applying Lemma \ref{prop_ineqsff} we get
\[
\int_\Sigma \phi\|A\|^2 e^{-f}d\mathrm{vol}_\Sigma=0.
\]
Hence
\begin{align*}
-\int_\Sigma \varphi\Delta_f \varphi+(\|A\|^2+1)\varphi^2e^{-f}d\mathrm{vol}_\Sigma&=
-\int_\Sigma \varphi^2 +\|A\|^2a^2+a\phi\|A\|e^{-f}d\mathrm{vol}_\Sigma\\
&= -\int_\Sigma \varphi^2 +\|A\|^2a^2 e^{-f}d\mathrm{vol}_\Sigma\\
&< 0.
\end{align*} 

Finally, assume that $\mathrm{Ind}_f(\Sigma)=m+2$. Note that, since $H$ belongs to $\mathcal{V}$ (see the proof of Theorem 9.36 in \cite{CoMi}) and $L_f H=-2H$, we get that $H$ does not change sign. To see this, assume by contradiction that $H$ changes sign. Then, by Theorem 9.36 in \cite{CoMi} it follows that the least negative eigenvalue $\lambda_1:=\lambda^{L_f}_1(\Sigma)$ of $L_f$ is strictly less than $-2$. Let $g$ be an $L^2(\Sigma_f)$ eigenfunction relative to  $-\lambda_1$. Since $\mathrm{Ind}_f(\Sigma)=m+2$ it follows by Proposition \ref{prop_eigenjacobi} that $g=aH+\phi$ for some $0\neq a\in\mathbb{R}$ and for some $\phi\in U$. Then $g\in \mathcal{V}$ and, as a consequence of Lemma 9.25 in \cite{CoMi}, $g$ does not change sign. Moreover,
\[
-(\lambda_1+1+\|A\|^2)g=\Delta_f g=a\Delta_f H+\Delta_f \phi=-\|A\|^2g+a H.
\]
In particular, $aH=-(\lambda_1+1)g$, contradicting the fact that $H$ changes sign. 

Since $H$ does not change sign and $\|A\|\in L^2(\Sigma_f)$, we can apply Theorem 8 in \cite{Rim} with the choices $u=\|A\|$, $v=H$, $a=\|A\|^2-1$ to conclude that $\|A\|=CH$ for some $0\neq C\in \mathbb{R}$ and that $|\nabla A|=|\nabla \|A\||$. These are the key geometric identities to prove our assertion. Indeed, reasoning as in the proof of Theorem 10.1 in \cite{CoMi}, we can proceed as follows.

Let $p\in\Sigma$, $\{e_i\}_{i=1}^{m}$ an orthonormal frame in $T_p\Sigma$ that diagonalizes $A_{p}$, then
\[
\langle Ae_{i}, e_{j}\rangle=\lambda_{i}\delta_{ij}.
\]
We hence have that
\begin{itemize}
	\item[(i)] For every $k$, there exists $\alpha_{k}$ such that $\nabla_{e_{k}}\langle Ae_{i}, e_{i}\rangle=\alpha_{k}\lambda_{i}$, $i=1\ldots,m$.
	\item[(ii)] If $i\neq j$ then $\nabla_{e_k}\langle Ae_{i}, e_{j}\rangle=0$. 
\end{itemize}
Since $\nabla A$ is fully symmetric, by Codazzi equations, (ii) implies
\begin{itemize}
	\item[($\tilde{\mathrm{ii}}$)]$\nabla_{e_{k}}\langle Ae_{i}, e_{j}\rangle=0$ unless $i=j=k$. 
\end{itemize}
If $\lambda_{i}\neq0$ and $j\neq i$, then $0=\nabla_{e_{j}}\langle Ae_{i}, e_{i}\rangle=\alpha_{j}\lambda_{i}$, so that $\alpha_{j}=0$.

In particular, if $\mathrm{rk}(A_{p})\geq 2$, then $\alpha_{j}=0$ for every $j$, and thus, by (i), $\left(\nabla A\right)_{p}=0$.

We now consider two cases, depending on the rank of $A$.

\textbf{Case 1.} \textit{There exists $p\in\Sigma$ such that $\mathrm{rk}(A_p)\geq2$}. Reasoning as in \cite{CoMi} we deduce that $\mathrm{rk}(A)$ is at least two everywhere and hence $\nabla A\equiv 0$ on $\Sigma$. According to a theorem of Lawson, \cite{Lawson}, we then obtain that $\Sigma$ is isometric to a cylinder $\mathbb{S}^k(\sqrt{k})\times\mathbb{R}^{m-k}$, for some $1\leq k\leq m$.

\textbf{Case 2.} $\mathrm{rk}(A_p)=1$. Since $|H|>0$, the remaining case to consider is when the rank of $A$ is exactly one at every point. Assume without loss of generality that $H>0$. In this case, reasoning as in Case 2 of the proof of Theorem 10.1 in \cite{CoMi}, we are able to conclude that $\Sigma$ is invariant under the isometric translations in the $(m-1)$--dimensional subspace spanned by a global orthonormal frame for $\mathrm{Ker}(A)$. Therefore $\Sigma$ is a product of a curve $\Gamma\subset\mathbb{R}^2$ and this $(m-1)$--dimensional subspace. The curve $\Gamma$ has to be a smooth complete self-shrinking curve in $\mathbb{R}^2$ with $H>0$ and finite $f$-length. In particular, as a consequence of Lemma 10.39 in \cite{CoMi}, $\Gamma$ is bounded and hence, since it is complete, it is a closed self-shrinking curve with $H>0$. Theorem A in \cite{AL} implies then that either $\Gamma$ is a unit circle or it is a member of the Abresch-Langer family $\{\Gamma_{p,n}\}$ described as follows. If $p, n$ are 
positive integers satisfying $1/2<p/n<\sqrt{2}/2$, there is, up to congruence, a unique curve $\Gamma_{p,n}$ having rotation index $p$ and with $2n$ vertices (i.e. critical points of the curvature $H$). Proposition 2.1 in \cite{EW} implies then that either $\Gamma$ is a unit circle or $n\geq 3$ and the $f$-stablity operator on $\Gamma$ has at least $4$ negative eigenvalues (counted with multiplicity). However this latter possibility cannot occur since, otherwise, we would have that $\mathrm{Ind}_f(\Sigma)=m-1+\mathrm{Ind}_f(\Gamma)>m+2$, contradicting our assumption. Hence $\Gamma$ must be a unit circle and this concludes the proof.   

\end{proof}
\begin{acknowledgement*}
The author has been supported by the ``Gruppo
Nazionale per l'Analisi Matematica, la Probabilit\`a e le loro
Applicazioni'' (GNAMPA) of the Istituto Nazionale di Alta Matematica
(INdAM). Finally, the author is deeply grateful to Stefano Pigola and Michele Rimoldi
for useful conversations during the preparation of the manuscript.
\end{acknowledgement*}

\bigskip

\bibliographystyle{amsplain}
\bibliography{bib_LowIndSS}
 
\end{document}